\documentclass[12pt,leqno]{article}
\usepackage{amssymb,amsthm}
\usepackage{amsmath,amssymb,color}
\newtheorem{thm}{Theorem}
\newtheorem{lem}[thm]{Lemma}

\numberwithin{equation}{section}

\def\BC{\mathbb C}
\def\BN{\mathbb N}
\def\BR{\mathbb R}

\def\cD{\mathcal D}

\def\rd{\mathrm d}
\def\rRe{\mathrm{Re}}

\def\Ga{\Gamma}
\def\De{\Delta}

\def\La{\Lambda}

\def\Om{\Omega}
\def\al{\alpha}
\def\be{\beta}

\def\ve{\varepsilon}
\def\te{\theta}
\def\ze{\zeta}
\def\ka{\kappa}
\def\la{\lambda}
\def\si{\sigma}
\def\vp{\varphi}

\def\f{\frac}

\def\ov{\overline}
\def\pa{\partial}
\def\un{\underline}

\def\wt{\widetilde}

\allowdisplaybreaks[4]

\title{Blow-up for time-fractional diffusion equations
with superlinear convex semilinear terms}

\author{Xinchi Huang\thanks{Institute for Innovative Research,
Tokyo Institute of Technology, 2-12-1 Ookayama, Meguro-ku,
Tokyo 152-8550, Japan. E-mail: {\tt huang.x.ag@m.titech.ac.jp}}\quad
Yikan Liu\thanks{Department of Mathematics, Kyoto University,
Kitashirakawa-Oiwakecho, Sakyo-ku, Kyoto 606-8502, Japan.
E-mail: {\tt liu.yikan.8z@kyoto-u.ac.jp}}\quad
Masahiro Yamamoto\thanks{Graduate School of Mathematical Sciences,
The University of Tokyo, Komaba, Meguro-ku, Tokyo 153-8914, Japan;
Honorary Member of Academy of Romanian Scientists, Ilfov, nr. 3, Bucuresti,
Romania; Correspondence member of Accademia Peloritana dei Pericolanti,
Palazzo Universit\`a, Piazza S. Pugliatti 1 98122 Messina Italy.
E-mail: {\tt myama@ms.u-tokyo.ac.jp}}}

\date{}

\begin{document}

\maketitle

\baselineskip 18pt

\begin{abstract}
This article is concerned with a semilinear time-fractional diffusion equation
with a superlinear convex semilinear term in a bounded domain $\Om$
with the homogeneous Dirichlet, Neumann, Robin boundary conditions 
and non-negative and not identically vanishing initial value.  
The order of the fractional derivative in time
is between $1$ and $0$, and the elliptic part is with time-independent 
coefficients.  We prove
\begin{enumerate}
\item[(i)] The solution with any initial value blows up if the eigenvalue $\la_1$ 
of the elliptic operator with the minimum real part is non-positive.
\item[(ii)] Otherwise, the solution blows up if a weighted $L^1$-norm of initial value
is greater than some critical value give by $\la_1$.
\end{enumerate}
We provide upper estimates of the blow-up times.
The key is a comparison principle for time-fractional ordinary differential 
equations. \medskip

\medskip

\end{abstract}

\section{Introduction and Main Results}

Let $d=1,2,3$ and $\Om\subset\BR^d$ be a bounded domain
with smooth boundary $\pa\Om$ and let $\nu=(\nu_1(x),\dots,\nu_d(x))$
be the unit outward normal vector to $\pa\Om$ at
$x\in\pa\Om$. 

We define
$$
-A v(x)=\sum_{i,j=1}^d\pa_i(a_{i j}(x)\pa_j v(x)) 
+\sum_{j=1}^d b_j(x)\pa_j v(x)+c(x)v(x),
$$
where $a_{i j}=a_{j i}\in C^2(\ov\Om)$, $b_j,c\in C(\ov\Om)$ for 
$1\le i,j\le d$ and there exists a constant $\ka>0$ such that 
$$
\sum_{i,j=1}^d a_{i j}(x)\xi_i\xi_j\ge\ka\sum_{j=1}^d\xi_j^2
$$
for all $x\in\ov\Om$ and $\xi_1,\dots,\xi_d\in\BR$.
Let 
$$
\pa_{\nu_A}v(x):=\sum_{i,j=1}^d a_{i j}(x)(\pa_i v(x))\nu_j(x),
\quad x\in\partial\Om.
$$

For $0<\al<1$, let $\rd_t^\al$ denote the classical Caputo derivative:
$$
\rd_t^\al f(t):=\int^t_0\f{(t-s)^{-\al}}{\Ga(1-\al)}
f'(s)\,\rd s,\quad f\in W^{1,1}(0,T).
$$
Here $\Ga(\,\cdot\,)$ denotes the gamma function.

In order to discuss semilinear time-fractional
diffusion equations in suitable function spaces, 
we extend the classical Caputo derivative
$\rd_t^\al$ as follows. First, 
we define the Sobolev-Slobodecki space $H^\al(0,T)$ 
with the norm $\|\cdot\|_{H^\al(0,T)}$:
$$
\|f\|_{H^\al(0,T)}:=\left(\|f\|^2_{L^2(0,T)}+
\int^T_0\!\!\!\int^T_0\f{|f(t)-f(s)|^2}{|t-s|^{1+2\al}}
\,\rd s\rd t\right)^{\f12}, \quad 0<\alpha<1
$$
(e.g., Adams \cite{Ad}). Furthermore, we set $H^0(0,T):=L^2(0,T)$ and
$$
H_\al(0,T):=
\left\{\!\begin{alignedat}{2}
& H^\al(0,T), & \quad & 0<\al<\f12,\\
&\left\{f\in H^{\f12}(0,T);\,\int^T_0\f{|f(t)|^2}t\,\rd t <\infty\right\},
& \quad & \al=\f12,\\
& \{f\in H^\al(0,T);\, f(0)=0\}, & \quad &\f12<\al\le1
\end{alignedat}\right.
$$
with the norms defined by 
$$
\|f\|_{H_\al(0,T)}:=
\left\{\!\begin{alignedat}{2}
&\|f\|_{H^\al(0,T)}, & \quad & \al\ne\f12,\\
&\left(\|f\|^2_{H^{\f12}(0,T)}+\int^T_0\f{|f(t)|^2}t\,\rd t\right)^{\f12}, 
& \quad & \al=\f12.
\end{alignedat}\right.
$$
Moreover, for $\be>0$, we set 
$$
J^\be f(t):=\int^t_0\f{(t-s)^{\be-1}}{\Ga(\be)}f(s)\,\rd s,
\quad0<t<T,\ f\in L^1(0,T).
$$
It was proved e.g.,\! in Gorenflo, Luchko and Yamamoto \cite{GLY} that
$J^\al:L^2(0,T)\longrightarrow H_\al(0,T)$ is an isomorphism for $\al\in(0,1)$.
Now we are ready to define the extended Caputo derivative
$$
\pa_t^\al:=(J^\al)^{-1},\quad\cD(\pa_t^\al)=H_\al(0,T).
$$
Henceforth $\cD(\,\cdot\,)$ denotes the domain of an operator 
under consideration.
Then it is proved that 
$\pa_t^\al$ is the minimum closed extension of $\rd_t^\al$ with
$\cD(\rd_t^\al):=\{v\in C^1[0,T];\, v(0)=0\}$ and
$\pa_t^\al v=\rd_t^\al v$ for $v\in C^1[0,T]$ satisfying
$v(0)=0$. As for the details, we can refer to
Gorenflo, Luchko and Yamamoto \cite{GLY}, 
Kubica, Ryszewska and Yamamoto \cite{KRY}, Yamamoto \cite{Y2023}.

In this article, we consider the following initial-boundary value problems 
for a semilinear time-fractional diffusion equation:
\begin{equation}\label{1.1}
\begin{cases}
\pa_t^\al(u-a)=-A u+f(u) &\mbox{in }\Om\times(0,T),\\
u=0 &\mbox{on }\pa\Om\times(0,T)
\end{cases}
\end{equation}
and
\begin{equation}\label{1.2}
\begin{cases}
\pa_t^\al(u-a)=-A u+f(u) &\mbox{in }\Om\times(0,T),\\
\pa_{\nu_A}u+\si u=0 &\mbox{on }\pa\Om\times(0,T),
\end{cases}
\end{equation}
where $f\in C^1[0,\infty)$, $f\ge0$ in $[0,\infty)$ and 
$a \ge 0$ in $\Omega$.  Moreover for simplicity, we assume 
$\sigma \in C^{\infty}(\partial\Omega)$, although we can relax
such regularity.

We define the domain of the operator $A$ by 
$$
\cD(A)=\begin{cases}
\{v\in H^2(\Om);\, v=0\mbox{ on }\pa\Om\} &\mbox{for \eqref{1.1}},\\
\{v\in H^2(\Om);\, \pa_{\nu_A}v+\si v=0\mbox{ on }\pa\Om\} 
&\mbox{for \eqref{1.2}}
\end{cases}
$$
and $\|v\|_{\cD(A)}:=\|v\|_{H^2(\Om)}$ for both cases $v=0$ and
$\pa_{\nu_A}v+\si v=0$ on $\pa\Om$.

The left-hand side of the time-fractional differential equation
in \eqref{1.1} and \eqref{1.2} means that $u(x,\,\cdot\,)-a(x)\in H_\al(0,T)$
for almost all $x\in\Om$. Especially for $\f12<\al<1$, noting that
$v\in H_\al(0,T)$ implies $v(0)=0$ by the trace theorem, we can
understand that the left-hand side means that $u(x,0)=a(x)$
in the trace sense with respect to $t$. Thus, this corresponds to the
initial condition for $\al>\f12$, while we do not need any
initial conditions for $\al<\f12$.
There are other formulations for initial-boundary value problems
for time-fractional partial differential equations
(e.g., Sakamoto and Yamamoto \cite{SY}, Zacher \cite{Za}), but
here we do not provide comprehensive references.

For each of \eqref{1.1} and \eqref{1.2}, we can prove that 
for each $a\in\cD(A)$,
we can find $T_a>0$ such that there exists a unique solution 
\begin{equation}\label{1.3}
u\in C([0,T_a];\cD(A))\quad\mbox{satisfying }
u-a\in H_\al(0,T_a;L^2(\Om))
\end{equation}
(e.g., Jin \cite{Jin}, Luchko and Yamamoto \cite[Theorem 5.1]{LY2}).
More precisely, the value $T_a$ is not dependent on the choice of $a$,
but on an upper bound of $\| a\|_{\cD(A)}$.
In this article, we are mainly concerned with the non-existence of
global solution in time to \eqref{1.1} or \eqref{1.2} within the class \eqref{1.3}.

In the case of $\al=1$, concerning the non-existence of
global solutions in time, there have been enormous works since
Fujita \cite{F}, and we can refer to a comprehensive
monograph by Quittner and Souplet \cite{QS}.  See also
Fujishima and Ishige \cite{FI}, Ishige and Yagisita \cite{IshiY}.

For $0<\al<1$, the time-fractional diffusion equation in \eqref{1.1} and 
\eqref{1.2}
is a possible model for describing anomalous diffusion
in heterogeneous media.
There are very rapidly increasing interests also on semilinear
time-fractional differential equations, and so here we refer to only
a few works: Ahmad et al.\! \cite{AAAKT},
Borikhanov, Ruzhansky and Torebek \cite{BRT23},
Floridia, Liu and Yamamoto \cite{FLY},
Ghergu, Miyamoto and Suzuki \cite{GMS23},
Hnaien, Kellil, and Lassoued \cite{HKL},
Kirane, Laskri and Tatar \cite{KLT}, Kojima \cite{K23},
Suzuki \cite{S21,S22}, Vergara and Zacher \cite{VZ},
Zhang and Sun \cite{ZS15}.

Our approach is based on the comparison of solutions to
initial value problems for time-fractional ordinary
differential equations, which generalizes \cite{FLY}.
Such a method can date back to Kaplan \cite{Kp} and see also Payne \cite{Pa}.
The work \cite{VZ} discusses stability, instability and
blow-up for more general semilinear time-fractional diffusion equations, while
in our article, we construct lower solutions for the blow-up
solution to \eqref{1.1} and \eqref{1.2}.
To the best knowledge of the authors, there are no publications
providing an upper estimate of the blow-up times.

Throughout this article, we assume:
\begin{gather}
f\in C^1[0,\infty),\quad f\ge0,\quad
f\mbox{ is convex},\label{1.4}\\
\mbox{there exist constants $c_0>0$ and $p>1$ such that 
$f(\xi)\ge c_0\xi^p$ for all $\xi\ge1$}.\label{1.5}
\end{gather}
By $f\ge0$, we can apply the comparison principle
(e.g., \cite{LY1,LY2}) to verify that $u\ge0$ in $\Om\times(0,T)$ if
$a\ge0$ in $\Om$, where $u$ is the solution to \eqref{1.1} or \eqref{1.2}. 
In this article, we consider only non-negative solutions.

It is reasonable that we assume \eqref{1.5} with $p>1$.
Indeed, if we assume that there exist constants $C_1>0$ and 
$0<p<1$ such that $0\le f(\xi) \le C_1\xi^p$ for all
$\xi \ge 0$, then we can prove the global existence of solution to
\eqref{1.1} and \eqref{1.2}, that is, for arbitrarily given $T>0$, there
exists a solution $u$ to each of \eqref{1.1} and \eqref{1.2}.
The proof relies on a usual argument by the fixed point theorem 
and is simiar to \cite[Theorem 2]{FLY}), and so we omit the details.

In this article we are concerned with the blow-up. We set
\begin{equation}\label{1.6}
T_a:=\sup\{t>0; \, \|u(\,\cdot\,,t)\|_{L^1(\Om)}<\infty\}
\end{equation}
and we call $T_a>0$ the blow-up time.
By the local existence of the solution (e.g., \cite{LY1}), we see that
$T_a>0$ exists. Meanwhile, by the definition we see that
\[
\limsup_{t\uparrow T_a}\|u(\,\cdot\,,t)\|_{L^1(\Om)}=\infty.
\]

For the statement of the main result, we further introduce
$$
(-A^*v)(x)=\sum_{i,j=1}^d\pa_i(a_{i j}(x)\pa_j v(x))
-\sum_{j=1}^d\pa_j(b_j(x)v(x))+c(x)v(x)
$$
with
$$
\cD(A^\ast)=\begin{cases}
\{v\in H^2(\Om);\, v=0\mbox{ on }\pa\Om\} & \mbox{for \eqref{1.1}},\\
\{v\in H^2(\Om);\, \pa_{\nu_A}v+(\si-\sum_{j=1}^d b_j\nu_j) v=0\mbox{ on }\pa\Om\} 
& \mbox{for \eqref{1.2}}.
\end{cases}
$$
Then it is known that $A^*$ has a principal eigenvalue $\la_1$
such that $\BR\ni\la_1\le\rRe\,\la$ for any other eigenvalue
$\la\in\BC$ of $A^*$, and the corresponding eigenfunction
does not change the sign. Hence we can choose $\vp_1(x)$ satisfying
\begin{equation}\label{1.7}
A^*\vp_1=\la_1\vp_1,\quad\vp_1>0\mbox{ in }\Om,\quad
\int_\Om\vp_1(x)\,\rd x=1
\end{equation}
and we can refer for example,  
to Pao \cite[Theorem \ref{thm1}.2 (p.97)]{Pao}.
In the case of the Dirichlet boundary condition, see
Evans \cite[Theorem 3 (p.361)]{E}, Gilbarg and Trudinger
\cite[Theorem 8.38 (p.214)]{GT}, and
Pao \cite[Theorem \ref{thm1}.2 (p.97)]{Pao}. 
In these references, the zeroth-order
coefficient $c(x)$ is assumed to be non-positive, but this condition is not
essential because we can replace the eigenvalue problem
$A^*\vp=\la\vp$ by $(A^*+2\|c\|_{C(\ov\Om)})\vp
=\wt\la\vp$, where $c(x)$ is replaced 
by $c(x)-2\|c\|_{C(\ov\Om)}$. We see that $\vp_1\in H^2(\Om)
\subset C(\ov\Om)$ by the Sobolev embedding, because the spatial dimensions
$d\le3$.

We assume that 
\begin{equation}\label{1.8}
a\in\cD(A),\quad a\ge0,\,\not\equiv0\quad\mbox{in }\Om
\end{equation}
and set
$$
a_0:=\int_\Om a(x)\vp_1(x)\,\rd x.
$$
Then by \eqref{1.7}, we see that $a_0>0$.
   
Now we are ready to state our main result.

\begin{thm}\label{thm1}
{\bf (i)} If $\la_1\le0,$ then there exists a constant $T=T_a>0$
such that a solution to each of \eqref{1.1} and \eqref{1.2}
within the class \eqref{1.3} exists for $0<t<T_a,$
and \eqref{1.6} holds. Moreover$,$ $T_a$ is bounded from above as
\[
T_a\le T_a^*:=\left\{(p-1)\Ga(2-\al)\,c_0a_0^{p-1}\right\}^{-\f1\al}.
\]

{\bf (ii)} If $\la_1>0,$ we further assume
\begin{equation}\label{eq-asp-a0}
a_0>\left(\f{\la_1}{c_0}\right)^{\f1{p-1}}.
\end{equation}
Then the same conclusion as in case {\rm(i)} holds and
\[
T_a\le T^*_a:=
\left\{(p-1)\Ga(2-\al)(c_0a_0^{p-1}-\la_1)\right\}^{-\f1\al}.
\]
Here $p$ and $c_0$ are the constants given in \eqref{1.5}.
\end{thm}

Theorem \ref{thm1} generalizes the result for the case $\al=1$
(i.e., the classical parabolic equation) which is found for example in 
\cite[Theorem 17.1 (p.104)]{QS}.
 
This article is composed of three sections.
In Section \ref{sec-pre}, we show three key lemmata.
In Section \ref{sec-proof}, we complete the proof of Theorem \ref{thm1}.

\section{Preliminaries}\label{sec-pre}

We will show the following three lemmata.

\begin{lem}\label{lem1}
Let $f\in L^2(0,T)$ and $c\in C[0,T]$.
Then there exists a unique solution $y\in H_\al(0,T)$ to
$$
\pa_t^\al y-c(t)y=f(t),\quad0<t<T.
$$
Moreover, if $f\ge0$ in $(0,T)$, then $y\ge0$ in $(0,T)$.
\end{lem}

\begin{proof}
The unique existence of $y$ is proved in
Kubica, Ryszewska and Yamamoto \cite[\S 3.5]{KRY} for example.
The non-negativity $y\ge0$ in $(0,T)$ follows from the same
argument in \cite{LY1}, which is based on the
extremum principle by Luchko \cite{L2009}.
\end{proof}

\begin{lem}\label{lem2}
Let $c_0$ and $a_0\ge0$ be constants and
$y-a_0,z-a_0\in H_\al(0,T)\cap C[0,T]$ satisfy
$$
\pa_t^\al(y-a_0)\ge-c_0y(t)+f(y(t)),\quad
\pa_t^\al(z-a_0)\le-c_0z(t)+f(z(t))\quad\mbox{in }(0,T).
$$
Then $y\ge z$ in $(0,T)$.
\end{lem}

\begin{proof}
We set
$$
\pa_t^\al(y-a_0)+c_0y(t)-f(y(t)) =:F\ge0,\quad
\pa_t^\al(z-a_0)+c_0z(t)-f(z(t)) =:G\le0.
$$
Since $y-a_0,z-a_0\in H_\al(0,T)\cap C[0,T]$,
we see that $F,G\in L^2(0,T)$. Setting
$$
\te:=y-z=(y-a_0)-(z-a_0)\in H_\al(0,T),
$$
we have
$$
\pa_t^\al\te(t)+c_0\te(t)-(f(y(t))-f(z(t)))=F-G\ge0
\quad\mbox{in }(0,T).
$$
We can further prove that
\begin{equation}\label{2.1}
\pa_t^\al\te(t)-(c(t)-c_0)\te(t)\ge0,\quad0<t<T,
\end{equation}
where
\begin{equation}\label{2.2}
c(t):=\left\{\!\begin{alignedat}{2}
&\f{f(y(t))-f(z(t))}{y(t)-z(t)}, & \quad & y(t)\ne z(t),\\
& f'(y(t)), & \quad & y(t)=z(t).
\end{alignedat}\right.
\end{equation}
Indeed, we set $\La:=\{t\in [0,T]\mid y(t)\ne z(t)\}$.
For $t_0\in\La$, we immediately see that
$c(t_0)\te(t_0)= f(y(t_0))-f(z(t_0))$.
For $t_0\notin\La$, that is,
$$
\te(t_0)=(y(t_0)-a_0)-(z(t_0)-a_0)=0,
$$
first we assume that there does not exist any sequence $\{t_n\}\subset
\La$ such that $t_n\to t_0$. Then there exists some small $\ve_0>0$ 
such that
$(t_0-\ve_0,t_0+\ve_0)\cap\La=\emptyset$. This means
$\te(t)=0$ for $t_0-\ve_0<t<t_0+\ve_0$ and thus 
$$
c(t)\te(t)=f'(y(t))\te(t)=0,\quad f(y(t))-f(z(t)) =0,
\quad t_0-\ve_0<t<t_0+\ve_0.
$$
Hence, we obtain $c(t_0)\te(t_0)=f(y(t_0))-f(z(t_0))$.

Next, assume that there exists a sequence $\{t_n\}\subset\La$ such that
$t_n\longrightarrow t_0\notin\La$ as $n\to\infty$.
By $t_n\in\La$, we have $y(t_n)\ne z(t_n)$ and 
$$
c(t_n)\te(t_n)=\f{f(y(t_n))-f(z(t_n))}{y(t_n)-z(t_n)}\te(t_n),
\quad n\in\BN.
$$
Since $y,z,\te\in C[0,T]$ and $\te(t_0)=0$,
we employ the mean value theorem to conclude
$$
\lim_{n\to\infty}\f{f(y(t_n))-f(z(t_n))}{y(t_n)-z(t_n)}\te(t_n)
= f'(y(t_0))\te(t_0)=0.
$$
Hence, again we arrive at $c(t_0)\te(t_0)=f(y(t_0))-f(z(t_0))$
also in this case.
Thus we have verified \eqref{2.1} with \eqref{2.2}.
Moreover, since $y,z\in C[0,T]$, we can verify that $c\in C[0,T]$.

Therefore, a direct application of Lemma \ref{lem1} to \eqref{2.1} yields
$\te\ge0$ in $(0,T)$ or equivalently $y\ge z$ in $(0,T)$.
Thus the proof of Lemma \ref{lem2} is complete.
\end{proof}

Finally we show

\begin{lem}\label{lem3}
Let $q>0$ be a constant. Then
\[
\rd_t^\al\left(\left(\f T{T-t}\right)^q\right)
\le\f q{\Ga(2-\al)\,T^\al}\left(\f T{T-t}\right)^{q+1}.
\]
\end{lem}

\begin{proof}
Recall the Maclaurin expansion of $(1+\ze)^\be$
with a constant $\be\in\BR$:
\[
(1+\ze)^\be=\sum_{k=0}^\infty\f{\be(\be-1)\cdots(\be-k+1)}{k!}
\ze^k,\quad|\ze|<1.
\]
Then putting $\ze=-\f t T$ and $\be=-q-1$ gives
\begin{equation}\label{eq-maclaurin}
\left(\f T{T-t}\right)^{q+1}=\left(1-\f t T\right)^{-q-1}
=\sum_{k=0}^\infty\f{(q+1)\cdots(q+k)}{k!}\left(\f t T\right)^k,
\quad0<t<T.
\end{equation}
Then by
\[
\f\rd{\rd t}\left(\left(\f T{T-t}\right)^q\right)
=\f q T\left(\f T{T-t}\right)^{q+1}=\f q T\sum_{k=0}^\infty
\f{(q+1)\cdots(q+k)}{k!}\left(\f t T\right)^k,
\]
we calculate
\begin{align*}
\rd_t^\al\left(\left(\f T{T-t}\right)^q\right)
& =\int_0^t\f{(t-s)^{-\al}}{\Ga(1-\al)}
\f\rd{\rd s}\left(\left(\f T{T-s}\right)^q\right)\rd s\\
& =\f q{\Ga(1-\al)\,T}\sum_{k=0}^\infty
\f{(q+1)\cdots(q+k)}{k!\,T^k}\int_0^t(t-s)^{-\al}s^k\,\rd s.
\end{align*}
Here by
\[
\int_0^t(t-s)^{-\al}s^k\,\rd s
=\f{\Ga(1-\al)\,k!}{\Ga(k+2-\al)}t^{k+1-\al},
\]
we have
\begin{align}
\rd_t^\al\left(\left(\f T{T-t}\right)^q\right)
& =\f{q\,t^{1-\al}}T\sum_{k=0}^\infty
\f{(q+1)\cdots(q+k)}{\Ga(k+2-\al)}\left(\f t T\right)^k\nonumber\\
& \le\f q{\Ga(2-\al)\,T^\al}\sum_{k=0}^\infty
\f{(q+1)\cdots(q+k)}{k!}\left(\f t T\right)^k,\label{eq-caputo-est}
\end{align}
where we estimate $\Ga(k+2-\al)$ from below as
\[
\Ga(k+2-\al)\ge\left.\begin{cases}
\Ga(2-\al), & k=0,\\
\Ga(k+1)=k!, & k\in\BN
\end{cases}\right\}\ge\Ga(2-\al)\,k!.
\]
Comparing the right-hand side of \eqref{eq-caputo-est} with
\eqref{eq-maclaurin} immediately yields the desired inequality.
\end{proof}

\section{Proof of Theorem \ref{thm1}}\label{sec-proof}

Recall that we assumed the existence of the solution $u$ to
\eqref{1.1} or \eqref{1.2} within the class \eqref{1.3}.\medskip

{\bf Step 1.} We set
$$
\eta(t):=\int_\Om u(x,t)\vp_1(x)\,\rd x
=\int_\Om(u(x,t)-a(x))\vp_1(x)\,\rd x+a_0,\quad0<t<T,
$$
where $a_0:=\int_\Om a(x)\vp_1(x)\,\rd x$.
Here we see $a_0>0$ from \eqref{1.7}--\eqref{1.8}.

By \eqref{1.8} and $\vp_1\in L^\infty(\Om)$,
we have $\int_\Om(u(x,t)-a(x))\vp_1(x)\,\rd x\in H_\al(0,T)$.
Fixing $\ve>0$ arbitrarily small, we see
$$
\eta(t)-a_0=\int_\Om(u(x,t)-a(x))\vp_1(x)\,\rd x\in H_\al(0,T-\ve)
$$
and hence
$$
\pa_t^\al(\eta(t)-a_0)=\int_\Om\pa_t^\al(u-a)(x,t)\vp_1(x)\,\rd x,
\quad0<t<T-\ve.
$$
Since $\pa_{\nu_A}u+\si u=\pa_{\nu_A}\vp_1+\si\vp_1=0$
or $u=\vp_1=0$ on $\pa\Om\times(0,T-\ve)$, Green's 
formula and the governing equation $\pa_t^\al(u-a)=-Au+f(u)$ yield
\begin{align*}
\pa_t^\al(\eta-a_0) & =-\int_\Om A u(x,t)\vp_1(x)\,\rd x
+\int_\Om f(u(x,t))\vp_1(x)\,\rd x\\
& =-\int_\Om u(x,t)A^*\vp_1(x)\,\rd x
+\int_\Om f(u(x,t))\vp_1(x)\,\rd x\\
& =-\la_1\int_\Om u(x,t)\vp_1(x)\,\rd x+\int_\Om f(u(x,t))\vp_1(x)\,\rd x\\
& =-\la_1\eta(t)+\int_\Om f(u(x,t))\vp_1(x)\,\rd x.
\end{align*}
Here we used $A^*\vp_1=\la_1\vp_1$ in $\Om$.

In terms of $\int_\Om\vp_1(x)\,\rd x=1$ by \eqref{1.7}
and the convexity of $f$ by \eqref{1.4},
we can apply Jensen's inequality to have
$$
\int_\Om f(u(x,t))\vp_1(x)\,\rd x\ge
f\left(\int_\Om u(x,t)\vp_1(x)\,\rd x\right)=f(\eta(t)),\quad0<t<T.
$$
Therefore, we obtain
\begin{equation}\label{3.1}
\pa_t^\al(\eta(t)-a_0)\ge-\la_1\eta(t)+f(\eta(t)),\quad0<t<T.
\end{equation}

{\bf Step 2.} This step is devoted to the construction of
a lower solution $\un\eta(t)$ satisfying
\begin{equation}\label{3.2}
\pa_t^\al(\un\eta(t)-a_0)(t)\le-\la_1\un\eta(t)+f(\un\eta(t)),
\quad0<t<T-\ve,\quad\lim_{t\uparrow T}\un\eta(t)=\infty.
\end{equation}
Owing to the assumption \eqref{1.5}, we know that
\eqref{3.2} is guaranteed by
\begin{equation}\label{eq-lower}
\pa_t^\al(\un\eta(t)-a_0)(t)\le-\la_1\un\eta(t)+c_0\,\un\eta^p(t),
\quad0<t<T-\ve,\quad\lim_{t\uparrow T}\un\eta(t)=\infty.
\end{equation}

As a candidate of a lower solution, we choose
$$
\un\eta(t):=a_0\left(\f T{T-t}\right)^q,\quad q>0
$$
Then Lemma \ref{lem3} implies
\begin{equation}\label{3.3}
\pa_t^\al(\un\eta(t)-a_0)=\rd_t^\al\un\eta(t)
=a_0\rd_t^\al\left(\left(\f T{T-t}\right)^q\right)
\le\f{a_0q}{\Ga(2-\al)\,T^\al}
\left(\f T{T-t}\right)^{q+1}.
\end{equation}
We divide the proof into 2 cases according to the sign of $\la_1$.\medskip

{\bf Case 1.} If $\la_1\le0$, then in terms of \eqref{3.3}
and $\eta(t)\ge0$, it turns out that \eqref{eq-lower} is guaranteed by
$$
\f{a_0q}{\Ga(2-\al)\,T^\al}\left(\f T{T-t}\right)^{q+1}
\le c_0\,\un\eta^p(t)=c_0a_0^p\left(\f T{T-t}\right)^{p q},\quad0<t<T
$$
or equivalently
$$
\f q{\Ga(2-\al)\,T^\al}\xi^{q+1}\le c_0a_0^{p-1}\xi^{p q},\quad\xi\ge1
$$
by changing $\xi:=\f T{T-t}$ for $0<t<T$.
This is guaranteed by
\[
q+1\le p q,\quad\f q{\Ga(2-\al)\,T^\al}\le c_0a_0^{p-1}.
\]
To minimize the lower bound of the choice of $T$,
we fix $q=\f1{p-1}$ and thus
\[
T\ge\left\{(p-1)\Ga(2-\al)\,c_0a_0^{p-1}\right\}^{-\f1\al}=:T_a^*.
\]
Therefore, the choice 
$$
\un\eta(t):=a_0\left(\f{T^*_a}{T^*_a-t}\right)^{\f1{p-1}}
$$
satisfies \eqref{3.2}. Then we can apply Lemma \ref{lem2} to
\eqref{3.1} and \eqref{3.2} and obtain
\begin{equation}\label{3.5}
\un\eta(t)\le\eta(t),\quad0<t<\min\{T^*_a,T_a\}.
\end{equation}

Now it remains to verify $T_a\le T_a^*$,
where $T_a$ was defined in \eqref{1.6}.
If $T_a>T_a^*$ instead, then we see
\[
\|u(\,\cdot\,,T_a^*)\|_{L^1(\Om)}<\infty
\]
by \eqref{1.6}. Moreover, \eqref{3.5} implies
$$
\un\eta(t)\le\eta(t),\quad0<t< T^*_a
$$
and hence $\lim_{t\uparrow T_a^*}\eta(t) =\infty$.
On the other hand, since
$$
\eta(t)\le\|\vp_1\|_{L^\infty(\Om)}\|u(\,\cdot\,,t)\|_{L^1(\Om)},
\quad0<t< T_a^*=\min\{T^*_a,T_a\},
$$
we have $\eta(T_a^*) <\infty$, which is a contradiction.
This concludes $T_a\le T_a^*$.\medskip

{\bf Case 2.} If $\la_1>0$, then similarly to Case 1,
it turns out that \eqref{eq-lower} is guaranteed by
$$
\f{a_0q}{\Ga(2-\al)\,T^\al}\left(\f T{T-t}\right)^{q+1}
\le-\la_1a_0\left(\f T{T-t}\right)^q
+c_0a_0^p\left(\f T{T-t}\right)^{p q},\quad0<t<T
$$
or equivalently
$$
\f q{\Ga(2-\al)\,T^\al}\xi^{q+1}\le c_0a_0^{p-1}\xi^{p q}
-\la_1\xi^q,\quad\xi\ge1.
$$
Choosing $q=\f1{p-1}$ again and dividing by $\xi^{q+1}$,
we obtain
\[
\f1{(p-1)\Ga(2-\al)\,T^\al}\le c_0a_0^{p-1}-\f{\la_1}\xi,
\quad\xi\ge1,
\]
which is guaranteed by
\[
\f1{(p-1)\Ga(2-\al)\,T^\al}\le c_0a_0^{p-1}-\la_1.
\]
Thanks to the assumption \eqref{eq-asp-a0},
we see $c_0a_0^{p-1}-\la_1>0$. Consequently, if
$$
T\ge\left\{(p-1)\Ga(2-\al)(c_0a_0^{p-1}-\la_1)\right\}^{-\f1\al}=:T^*_a,
$$
then $\un\eta(t)\le\eta(t)$ for $0<t<\min\{T_a,T^*_a\}$.
The verification of $T_a\le T^*_a$ is identical to that in Case 1.
Thus the proof of Case 2 and so the proof of Theorem \ref{thm1} are completed.\bigskip

{\bf Acknowledgements }
The second author is supported by Grant-in-Aid for Early-Career Scientists
22K13954, Japan Society for the Promotion of Science (JSPS).
The third author is supported by Grant-in-Aid for Scientific Research (A)
20H00117 and Grant-in-Aid for Challenging Research (Pioneering) 21K18142, JSPS.


\end{document}